[1994/12/01]
\documentclass[reqno,12pt]{amsart}
\title{Feedback theory approach to positivity and stability of evolution equations}
\author{Abed Boulouz, Hamid Bounit and Said Hadd}
\date{September 14, 2020}
\address{Department of Mathematics, Faculty of Sciences, Ibn Zohr University, Hay Dakhla, BP8106, Agadir, Morocco; abed.boulouz@edu.uiz.ac.ma, h.bounit@uiz.ac.ma, s.hadd@uiz.ac.ma}
\thanks{}

\usepackage{amsmath,amsthm}
\usepackage{amsfonts}
\usepackage{amssymb}
\usepackage{latexsym}
\usepackage{a4wide}

\chardef\bslash=`\\ 

\newtheorem{theorem}{Theorem}[section]
\newtheorem{lemma}[theorem]{Lemma}
\newtheorem{corollary}[theorem]{Corollary}
\newtheorem{proposition}[theorem]{Proposition}

\theoremstyle{definition}
\newtheorem{definition}[theorem]{Definition}
\newtheorem{example}[theorem]{Example}

\theoremstyle{remark}
\newtheorem{remark}{Remark}

\def\al{\alpha}

\def\la{\lambda}

\def\calL{{\mathcal{L}}}

\def\calA{{\mathcal{A}}}

\begin{document}
\maketitle
\renewcommand{\sectionmark}[1]{}
\begin{abstract}
In this paper, we study the positivity and (uniform) exponential stability of a large class of perturbed semigroups. Our approach is essentially based on the feedback theory of infinite-dimensional linear systems. The obtained results are applied to the stability of hyperbolic systems including those with a delay at the boundary conditions.
\end{abstract}
\section{Introduction}\label{sec:1}

In this paper, we introduce a new approach based on feedback theory of infinite-dimensional linear systems to prove positivity and stability of perturbed semigroups in Banach lattices. We note that  positivity of semigroups plays a key role in the study of the asymptotic behavior of Cauchy problems. The latter model problems in many areas such as physics and biology(see e.g \cite{Banasiak Arlotti},\cite{Boulouz Bounit Hadd},\cite{Gwizdz Tyran-Kaminska},\cite{Hadd Rhandi},\cite{Rhandi A}, \cite{Rhandi},\cite{Zhang Prieur} and references therein).

In this work, we consider the following abstract perturbed boundary value problem
\begin{equation}\label{absract boundary}
\left\{
\begin{array}{ll}
\dot{z}(t)= A_m z(t), & \hbox{$t\geq 0$;} \\
Gz(t)=Mz(t), & \hbox{$t\geq 0$;} \\
z(0)=z. & \hbox{}
\end{array}
\right.
\end{equation}
Hereafter, $A_m:Z\longrightarrow X $ is a linear closed operators in a Banach space $X$ with maximal domain $D(A_m):=Z\hookrightarrow X,$ and $G,M:Z\longrightarrow U$ are linear operators (here $U$ is an another Banach space). Observe that the abstract boundary system  \eqref{absract boundary} can be converted into the following Cauchy problem
\begin{equation}\label{perturbed Cauchy}
\left\{
\begin{array}{ll}
\dot{z}(t)= \calA z(t), & \hbox{$t\geq 0$;} \\
z(0)=z, & \hbox{}
\end{array}
\right.
\end{equation}
where
\begin{align}\label{generator}
\calA:=A_m,\quad D(\calA):=\{z\in Z: Gz=Mz\}.
\end{align}
Types of problems like  \eqref{absract boundary} were first studied by Greiner \cite{Greiner} and Salamon \cite{Salamon} in the case when $M$ is bounded, i.e. $M\in\mathcal{L}(X,U)$. We mention that the case of unbounded perturbation $M:Z\subsetneq X\longrightarrow U$ was recently  studied in \cite{Hadd Manzo Ghandi}, where feedback theory of regular linear systems developed by \cite{Staffans} and \cite{Weiss} is extensively used.  An example of \eqref{absract boundary} with unbounded $M$, one can consider neutral equations investigated in \cite{Hadd Rhandi}. In fact, in \cite{Hadd Rhandi}, the authors proved the well-posedness and positivity of such equations. In the present work, we  prove the positivity and stability of the semigroup solutions of the general abstract problem \eqref{absract boundary}.

More recently, the authors of \cite{Gwizdz Tyran-Kaminska} consider a particular case of \eqref{absract boundary} and prove that the positivity in the space $L^1$ using background in \cite{Arendt} and \cite{Thieme}. The paper \cite{Gwizdz Tyran-Kaminska} extended the well-known results stated by  Desch \cite{Desch} and Voigt \cite{Voigt} in AL-spaces. Similar results in AM-spaces was proved in \cite{Batkai Jacob Voigt Wintermayr}. We note that all these results do not include reflexive spaces. This why we will introduce a new approach based on feedback theory of regular linear systems to treat the case of general Banach lattices.

As an application, we present an important class of systems called hyperbolic systems. We mention that there  is a huge literature dealing with linear and non-linear hyperbolic systems (see e.g \cite{Bounit, Barbu,Bastin coron,Caldeira Prieur,Castilo,Diagne,Neves,Peralta,Peralta1,coron bastin novel} and references therein). In particular, we are interested in the following linear hyperbolic partial differential equation with positive boundary conditions
\begin{align}\label{positive hyperbolic}
\begin{cases}
{\displaystyle\frac{\partial y(x,t)}{\partial t}}+D {\displaystyle\frac{\partial y(x,t)}{\partial t}}=0,&x\geq 0,x\in [0,1],\\
y(0,t)=K y(1,t)&t\geq 0,\\
u(x,0)=u_0(x),&x\in [0,1],
\end{cases}
\end{align}
where $y(t,x)=(y_1(t,x),y_2(t,x),...,y_n(t,x))$, $D=diag\{d_1,d_2,...,d_n\}$ with the velocities $d_i>0$, for $i\in \{1,2,...,n \}$ and $K=(K_{ij})_{i,j=1,...,n}\in \mathbb{R}_+^{n\times n}$. We recall that the system \eqref{positive hyperbolic} is called positive if the trajectories of this system starting from any positive initial conditions remain positives.  We mention that the authors of  \cite{Zhang Prieur}  have established the positivity and exponential stability of system \eqref{positive hyperbolic}.  More precisely, they gave  necessary and sufficient conditions which assure exponential stability under an extra-condition on velocities. Their method is based on a strict linear Lyapunov function. Our contribution here is to investigate the well-posedness of hyperbolic system \eqref{positivity of hyperbolic} using the theory of infinite-dimensional linear systems. Due to the riche properties of positive semigroups, we are able to avoid the extra assumption used in \cite{Zhang Prieur} to  prove the exponential stability for the system \eqref{positive hyperbolic}. However, a delay is frequently appeared in such systems due to a time lag between two actions. For this reason, we consider the hyperbolic system \eqref{positive hyperbolic} with boundary conditions depends of the history of the system (see Section 5). More precisely, we deal with the following system
\begin{align}\label{delay heperbolic 0}
\begin{cases}
\vspace{0.1cm}
\displaystyle\frac{\partial y(x,t)}{\partial t}+D\displaystyle\frac{\partial y(x,t)}{\partial x}=0,& 0<x<1,t\geq 0,\\
y(0,t)=\displaystyle\int_{0}^{1}\int_{-1}^{0}d\mu(\theta)y_t(x,\theta)dx,&t\geq 0,\\
y(x,\theta)=\varphi(\theta),&0<x<1, -1\leq \theta\leq 0,\\
y(x,0)=y_0(x),&0<x<1.
\end{cases}
\end{align}
where $\mu$ is a function of bounded variation. We will see that the system \eqref{delay heperbolic 0} admits a unique positive solution. Moreover, we establish the necessary and sufficient conditions on the uniform exponential stability of \eqref{delay heperbolic 0}. In particular, we prove that the stability of \eqref{delay heperbolic 0} is extremely responsive to small delays.

The paper is organized as follows. In Section 2, we collect some results on infinite-dimensional linear systems. In Section 3, we introduce results on positivity and stability of closed loop systems. Section 4 is devoted to study positive hyperbolic systems. The last section deals with hyperbolic systems with delay boundary conditions.

\section{Preliminaries}
In this section, we collect from \cite{Salamon,Staffans,Tucsnak Weiss,Weiss,Hadd Manzo Ghandi} some results on the feedback theory of regular linear systems combined with results on perturbed boundary value problems.

Here, the operators $A_m,G$ and $M$ are as in the introductory section. In addition, we consider the following standard hypotheses introduced in Greiner \cite{Greiner} and Salamon \cite{Salamon}.
\begin{itemize}
	\item[($\mathbf{H1}$)] The operator
	\begin{align}\label{initial generator}
	A:=A_m,\quad D(A)=\{z\in Z: Gz=0\}
	\end{align}
	generates a semigroup $(T(t))_{t\geq 0}$ on $X$. Here $A_m$ is a maximal operator as defined in \eqref{absract boundary}.
	\item[($\mathbf{H2}$)] The operator $G:Z\longrightarrow U$ is surjective.
\end{itemize}
With these hypothesis, Greiner \cite{Greiner} showed that the restriction
\begin{align*}
G_{|\ker(\la-A_m)}:\ker(\la-A_m)\longrightarrow  U.
\end{align*}
is invertible for $\la\in\rho(A)$ and its inverse
\begin{align*}
D_\la:U\longrightarrow \ker(\la-A_m) \subset Z\subset X
\end{align*}
is bounded from $U$ to $X$ (called Dirichlet operator).

We denote by $X_{-1,A}$ (sometimes we simply write $X_{-1}$) the completion of $X$ with respect to the norm $\|x\|_{-1}:=\|R(\la,A)x\|$, $x\in X$ for some (hence all) $\la\in\rho(A)$. We mention that the semigroup $(T(t))_{t\geq 0}$ can be extended to a strongly continuous semigroup $(T_{-1}(t))_{t\geq 0}$ on $X_{-1,A}$, whose generator $A_{-1}:D(A_{-1})=X \longrightarrow X_{-1,A}$ is the extension of $A$ to $X$ (we refer to Chapter II of \cite{Engel Nagel}).

Let us define the control and observation operators as follows:
\begin{align}\label{opertaorB}
B&:=(\lambda-A_{-1})D_{\lambda}\in \calL(U,X_{-1,A}) \quad  \text{for } \la\in \rho(A), \\  C&:=M_{|D(A)} \in \calL(D(A),U)\nonumber.
\end{align}

Define the control (or input) map associated to $A$ and $B$
\begin{equation*}
\Phi_{t}u:=\int_{0}^{t}T_{-1}(t-s)Bu(s)ds, \quad u\in L^{p}([0,+\infty),U),\quad t\geq 0,
\end{equation*}
which belongs to $X_{-1,A}$. By using an integration by parts, for any
\begin{align*}
u\in W^{2,p}_{0,loc}([0,+\infty),U):=\left\{v\in W^{2,p}_{loc}([0,+\infty),U):v(0)=0  \right\},
\end{align*}
\begin{align*}
\Phi_t u=D_0 u(t)-\int_{0}^{t}T(t-s)D_0 u'(s)ds \in Z,\quad t\geq 0.
\end{align*}
This allows us to define the input-output operator
\begin{align*}
(\mathbb{F} u)(t):= M \Phi_t u, \quad u\in W^{2,p}_{0,loc}([0,+\infty),U), a.e. t\geq 0.
\end{align*}
The following definition gives sense to the solution of the system \eqref{absract boundary} (see also \cite{Staffans} and \cite{Weiss}).
\begin{definition}\label{well-posed-defi}
	The system $\Sigma:=(A,B,C)$ is well-posed if
	\begin{itemize}
		\item[(i)] $B$ is an admissible control operator for $A$. Namely, there exists $t_{0}>0$ such that
		\begin{align*}
		\Phi_{t_{0}} u \in X,\quad \forall u\in L^{p}(\mathbb{R}^+,U).
		\end{align*}
		\item[(ii)] $C$ is an admissible observation operator for $A$. Namely, there exist some constants $\alpha>0$  and $\gamma:=\gamma(\alpha)>0$ such that
		\begin{align*}
		\int_{0}^{\alpha}\|CT(t)x\|^p dt \leq \gamma^p \|x\|^p, \quad \forall x\in D(A),
		\end{align*}
		\item[(iii)] There exist constants $\tau>0$ and $\kappa:=\kappa(\tau)>0$ such that
		\begin{align*}
		\| \mathbb{F} u\|_{L^p([0,\tau],U)}\leq \kappa\|u\|_{L^p([0,\tau],U)},\quad u\in W^{2,p}_{0,loc}([0,+\infty),U).
		\end{align*}		
	\end{itemize}
\end{definition}
By using the closed graph theorem one can see that the point ${\rm(i)}$ in Definition \ref{well-posed-defi} implies that
\begin{align*}
\Phi_t\in\mathcal{L}(L^{p}(\mathbb{R}^+,U),X),\qquad\forall t\ge 0.
\end{align*}
On the other hand, the assertions ${\rm(ii)}$ and ${\rm(iii)}$ in Definition \ref{well-posed-defi} hold also for any $\al>0$ and $\tau>0,$ respectively.

In \cite{Weiss}, Weiss introduce a subclass of a well-posed linear system $\Sigma=(A,B,C)$ as follows
\begin{definition}
	A well-posed system $\Sigma:=(A,B,C)$ is  called regular (with zero feedthrough) if for any $u\in U$, the following limit
	\begin{align*}
	\lim_{\tau \to 0}\int_{0}^{\tau}(\mathbb{F}(\chi_{\mathbb{R}^+}\cdot u))(\sigma) d\sigma=0
	\end{align*}
	exists in $U$, where $\chi_{\mathbb{R}^+}$ is a constant function equals to $1$ on $\mathbb{R}^+$.
\end{definition}
The following theorem shows a characterization of the regularity of a system (see \cite{Weiss} and \cite[Theorem 5.6.5]{Staffans})
\begin{theorem}\label{characterization regular}
	Let $\Sigma:=(A,B,C)$ be a well-posed system. The following statements are equivalent.
	\begin{itemize}
		\item $\Sigma$ is regular.
		\item	$Range (D_\la) \subset D(C_\Lambda) $, for $\la\in\rho(A)$, where the operator $C_\Lambda$ is the Yosida extension of $C$ defined by
		\begin{align*}
		C_\Lambda x &:=\lim_{\la\to +\infty}C\la R(\la,A)	x \\
		D(C_\Lambda)&:=\left\{x\in X:\lim_{\la\to +\infty}C\la R(\la,A)	x \text{ exists in } U \right\}.
		\end{align*}
		\item For any $u\in U$, $H(\lambda)u$ has a limit when $\lambda\to \infty$ (equals to $0$), where $H$ is the transfer function associated to $\Sigma$.
	\end{itemize}
	In this case, the transfer function $H$ is given by
	\begin{align*}
	H(\lambda):=MD_{\lambda}, \quad Re\lambda>w_{0}(A).
	\end{align*}
\end{theorem}
\begin{definition}
	Let $\Sigma=(A,B,C)$ be a regular linear system. The identity operator $I:U\longrightarrow U$ is called an admissible feedback operator for $\Sigma$ if $I-\mathbb{F}$ has an inverse in $\mathcal{L}(L^p([0,\tau],U))$ for some $\tau>0$.
\end{definition}
It is worth noting that in the Hilbert setting, we may use function transfers instead of input-output operators for the definition of admissible feedback operators.

In the sequel, we consider the following assumption:
\begin{itemize}
	\item[($\mathbf{A1}$)] Assume that the conditions $(H1)$ and $(H2)$ is satisfied and that  $\Sigma=(A,B,C)$ is a regular linear system with $I:U\longrightarrow U$ is admissible feedback operator for $\Sigma$.
\end{itemize}

The following theorem presents a large class of perturbations called Staffans-Weiss, due to Weiss \cite{Weiss} (in Hilbert cases) and Staffans (in Banach cases).
\begin{theorem}
	Let the assumption ${\bf(A1)}$ be satisfied. Let us define the operator
	\begin{align}\label{Staffans Weiss}
	A^{cl}:=A_{-1}+BC_{\Lambda},\quad D(A^{cl}):=\left\{x\in D(C_{\Lambda}):(A_{-1}+BC_{\Lambda})x\in X\right\}.
	\end{align}	
	Then $A^{cl}$ generates a $C_0$-semigroup $(T^{cl}(t))_{t\geq 0}$ on $X$ satisfying $T^{cl}(s)x\in D(C_{\Lambda})$ for all $x\in X$, a.e $s\geq 0$ and
	\begin{align}\label{semigroup Staffans Weiss}
	T^{cl}(t)x=T(t)x+\int_{0}^{t}T_{-1}(t-s)BC_{\Lambda}T^{cl}(s)xds,
	\end{align}
	for all $x\in X$ and $t\geq 0$.
	
\end{theorem}

Let the assumption ${\bf(A1)}$ be verified, it is shown in \cite{Hadd Manzo Ghandi} that $Z\subset D(C_{\Lambda})$ and $C_{\Lambda}x=Mx$, for $x\in Z$. Moreover  the operator $\calA$ (see \eqref{generator}) coincides with the operator defined in \eqref{Staffans Weiss}. Hence, $(\calA,D(\calA))$ generates a $C_0$-semigroup denoted also by $(T^{cl}(t))_{t\geq 0}$ which is given by \eqref{semigroup Staffans Weiss}.

In addition, for $\la\in \rho(A)$ we have
\begin{equation}\label{Res 1}
\la\in \rho(\calA) \Longleftrightarrow 1\in \rho(MD_\la)\Longleftrightarrow 1\in\rho(D_{\la}M).
\end{equation}
In this case
\begin{equation}\label{Res 2}
R(\la,\calA)=(I-D_\la M)^{-1} R(\la,A).
\end{equation}

\section{Positivity and exponential stability for perturbed boundary value systems}
In this section, we assume that the condition ${\bf(A1)}$ holds. Now, in order to discuss the positivity  of the semigroup generated by the operator $\mathcal{A}$ defined by \eqref{generator}, we assume that the spaces $X$ and $U$ are Banach lattices (see e.g. \cite{Batkai Kramer Rhandi},\cite{Nagel},\cite{Schaefer} for the definition). Hereafter, the uniform growth bound $\omega_0(A)$, the growth bound $\omega_1(A)$, the spectral bound $s(A)$, and the radius spectral $r(A)$ of a generator $A$ on $X$ are defined as follows:
\begin{align*}
\omega_0(A)&:=\inf\left\{\omega\in \mathbb{R}: \exists M_\omega>0: \|T(t)\|\leq M_\omega e^{\omega t},\forall t\geq 0   \right\},\\
\omega_1(A)&:=\inf\left\{\omega\in \mathbb{R}:\forall x\in D(A), \exists M_{\omega,x}>0: \|T(t)x\|\leq M_{\omega,x} e^{\omega t},\forall t\geq 0   \right\},\\
s(A)&:=\sup\{Re\lambda:\lambda\in \sigma(A)\}  \text{ and } \\
r(A)&:=\sup\{|\lambda|:\lambda\in \sigma(A)\}.
\end{align*}
The semigroup $(T(t))_{t\geq 0}$ is  said to be :
\begin{itemize}
	\item[(i)]  uniformly exponentially stable if $\omega_0(A)<0$.
	\item[(ii)] exponentially stable if $\omega_1(A)<0$.	
\end{itemize}
Clearly, for semigroups, the exponential stability does not imply the uniform stability. The goal of this section is to study the (uniformly) exponentially stable semigroup by using the theory of positive semigroups. We first recall that a semigroup $(T(t))_{t\geq 0}$ is positive (i.e. $T(t)$ positive for all $t\geq 0$) if and only if $R(\lambda,A)\geq 0$ for all $\lambda>s(A)$ (see e.g. \cite{Arendt}). Moreover, for a positive semigroup $(T(t))_{t\geq 0}$ with generator $A$, we have $\lambda_0\in \rho(A)$ and $R(\lambda_0,A)\geq 0$ if and only if $\lambda_0>s(A)$. The intrinsic consequence of positive semigroup is that the exponential stability is characterized by $s(A)<0$ (i.e $\omega_1(A)=s(A)$). If we choose the Banach lattice $X$ to be the Lebesgue space $L^p(\Omega),$ then $\omega_0(A)=\omega_1(A)=s(A)$ (see \cite{Weis}). In this case, the uniform exponential stability is also characterized by  $s(A)<0$.
We shall now state the following assumption:
\begin{itemize}
	\item[($\mathbf{A2}$)] The operators $M$, $R(\la,A)$ and $D_{\lambda}$ are positive operators for all $\lambda>s(A)$.
	
\end{itemize}
We mention that the control operator $B$ defined in \eqref{opertaorB} is positive if and only if the Dirichlet operator $D_{\la}$ is positive for all $\la>s(A)$ (see Remark 2.2 of \cite{Batkai Jacob Voigt Wintermayr}).\\
The following lemma,  which constitutes a general version of Theorem 3.1 of \cite{Arendt}, Theorem 1.1 of \cite{Voigt} and Theorem 3.2 \cite{Batkai Jacob Voigt Wintermayr}, plays a vital role in supporting our results.
\begin{lemma}\label{general version}
	Let the assumptions ${\bf(A1)}$-${\bf(A2)}$ be satisfied. Then the following conditions are equivalent :
	\begin{itemize}
		\item[i)] $r(MD_\la)<1$.
		\item[ii)] $\lambda\in\rho(\calA)$ and $R(\lambda,\calA)\geq 0$.
	\end{itemize}
	In addition, if one of these conditions is satisfied then
	\begin{equation}\label{Res mono}
	R(\lambda,\calA)\geq R(\lambda,A)\geq 0.
	\end{equation}
\end{lemma}
\begin{proof}
	Let $\la>s(A)$ and assume that $r(M D_\la)<1$ then $1\in \rho(M D_{\la})$ and
	\begin{align*}
	(I-MD_\la)^{-1}=\sum_{n\geq 0}(MD_\la)^{n}.
	\end{align*}
	By using the positivity of $MD_\la$, we have 	$(I-MD_\la)^{-1}\geq 0$.\\
	Thanks to \eqref{Res 1} and \eqref{Res 2}, we obtain $\la\in\rho(\calA)$ and
	\begin{align*}
	R(\la,\calA)& =(I-D_\la M)^{-1}R(\la,A)\\
	&=R(\la,A)+D_\la (I-MD_{\la})^{-1}MR(\la,A) \\
	&\geq R(\la,A)\geq 0.
	\end{align*}
	Conversely, let $\la\in \rho(\calA)$ and $R(\la,\calA)\geq 0$. Thus $\la\in \rho(\calA_{-1})=\rho(\calA)$ and $R(\la,\calA_{-1})\geq 0$ (see Definition 2.1 and Remark 2.2 of \cite{Batkai Jacob Voigt Wintermayr}). By virtue of \eqref{Res 1} and \eqref{Res 2}, we obtain  $1\in\rho(MD_\la)$ and
	\begin{align*}
	(I-MD_\la)^{-1} &= I+M(I-D_\la M)^{-1} D_\la \\
	&=I+MR(\la,\calA_{-1})B \geq 0,
	\end{align*}
	which implies $r(MD_\la)<1$.
\end{proof}
We now establish the positivity of the semigroup $(T^{cl}(t))_{t\geq 0}$ generated by $\calA$.
\begin{theorem}\label{Positivity}
	Let the assumptions ${\bf(A1)}$-${\bf(A2)}$ be satisfied. If there exists $\la_0>s(A)$ such that $r(MD_{\la_0})<1$  then $\calA$ generates a positive semigroup $(T^{cl}(t))_{t\geq 0}$ on $X$.
\end{theorem}
\begin{proof}
	Let $\la_0>s(A)$ such that $r(MD_{\la_0})<1$. Due to Lemma \ref{general version}, $R(\la_0,\calA)$ is positive. By using the resolvent equation we obtain
	\begin{align*}
	MD_{\la_0} - MD_\mu =(\mu - \la_0)M R(\la,A)D_{\la_0}\geq 0,
	\end{align*}	
	for $\mu>\la_0>s(A)$. This implies that the function $\la \longmapsto MD_\la $ is decreasing on $(s(A),+\infty)$. Therefore
	\begin{equation*}
	r(MD_\mu)\leq r(MD_{\la_0})<1,\quad \quad \mu>\la_0>s(A).
	\end{equation*}
	From Lemma \ref{general version}, we deduce
	\begin{align*}
	R(\mu,\calA) \geq 0, \quad \text{ for all } \mu>\la_0.
	\end{align*}
	It follows that $\calA$ generates a positive semigroup $(T^{cl}(t))_{t\geq 0}$ on $X$.
\end{proof}

It is worth noting that the authors of \cite{Hadd Rhandi} treat the positivity  of the example of neutral equations. They used similar conditions as in Theorem \ref{Positivity}.

In the following result, we introduce the importance of positivity in the analysis of stability.
\begin{proposition}\label{expo stab}
	Let the assumptions ${\bf(A1)}$-${\bf(A2)}$ be satisfied and $s(A)<0$. If $r(MD_0)<1$ then $\calA$ generates an exponentially stable semigroup $(T^{cl}(t))_{t\geq 0}$ in $X$.
\end{proposition}
\begin{proof}
	Let $s(A)<0$. From Lemma \ref{general version}, one can see that if $r(MD_0)<1$ then $s(\calA)<0$. It follows from  Theorem \ref{Positivity} that $\mathcal{A}$ generates a positive semigroup in $X$ hence $s(\mathcal{A})=\omega_1(\mathcal{A})<0$ (i.e. $(T^{cl}(t))_{t\geq 0}$ is exponentially stable).
\end{proof}
The following result shows the necessary and sufficient conditions on the exponential stability  of semigroup $(T^{cl}(t))_{t\geq 0}$ generated by $\mathcal{A}$.
\begin{theorem}\label{nece suff exp stablility}
	Let the assumptions ${\bf(A1)}$-${\bf(A2)}$ be satisfied and $s(A)<0$. Assume that $\calA$ generates a positive semigroup $(T^{cl}(t))_{t\geq 0}$ on $X$. Then $\calA$ generates an exponentially stable semigroup $(T^{cl}(t))_{t\geq 0}$ if and only if $r(MD_0)<1$.
\end{theorem}
\begin{proof}
	Let $\mathcal{A}$ generate a positive semigroup $(T^{cl}(t))_{t\geq 0}$. By using \eqref{Res 1}, \eqref{Res 2} and the properties of positive semigroups, we obtain the following equivalence
	\begin{align*}
	s(\calA)<0&\Longleftrightarrow 0\in \rho(\calA) \text{ and } R(0,\calA)\geq 0\\
	& \Longleftrightarrow 1\in \rho(D_{0}M) \text{ and } (I-D_{0}M)^{-1}\geq 0\\
	& \Longleftrightarrow 1\in \rho(MD_{0}) \text{ and } (I-MD_{0})^{-1}\geq 0\\
	& \Longleftrightarrow r(MD_{0})<1.
	\end{align*}
	This ends the proof.	
\end{proof}
\begin{remark}\label{remark stab}
	It should be noted that if we choose the state space $X$ to be a $L^p$-space or AL-space or Hilbert space then the results of Proposition \ref{expo stab} and Theorem \ref{nece suff exp stablility} are true for the uniform case (i.e. uniform exponential stability).
\end{remark}
\begin{example}
	Consider the following partial differential equation governed by transport equations
	\begin{equation}\label{transport 1}
	\begin{cases}
	\displaystyle\frac{\partial u(t,s)}{\partial t}=\rho \displaystyle\frac{\partial u(t,s)}{\partial s},&(t,s)\in \mathbb{R}^+\times [0,1];\\
	u(t,1)=k u(t,0),&t\in \mathbb{R}^+;\\
	u(0,\cdot)=u_0\in L^p [0,1],&
	\end{cases}
	\end{equation}
	where $k$ is a positive  constant and $\rho \in (0,+\infty)$ is the propagation velocity. This model describes certain physical phenomena such as networks, heat exchangers, chemical reactors. We set
	\begin{equation*}
	A_m:=\rho \frac{d}{ds},\quad Z=W^{1,p}(0,1).
	\end{equation*}
	It is well-known that the operator $A:=A_m$	with domain $D(A)=\{f\in Z:f(1)=0\}$ generates a nilpotent translation semigroup (hence a positive and exponentially stable semigroup) on $L^{p}[0,1]$. Clearly, the Dirichlet operator is given by
	\begin{equation*}
	(D_\la u)(x)=e^{\frac{\la}{\rho} (x-1)}u,\quad u\in \mathbb{R}, x\in [0,1].
	\end{equation*}
	Due to Example 2.3.4 of \cite{Staffans} and Lemma 3 of \cite{Boulouz Bounit Hadd}, The triple $(A,-A_{-1} D_0,k \delta_0 )$  is a regular linear system with $I:\mathbb{R}\longrightarrow \mathbb{R}$ is an admissible feedback operator. Thus, the solution of transport equation \eqref{transport 1} is positive and it is uniformly exponentially stable if and only if $k\in [0,1)$.
\end{example}
\begin{example}
	Consider the following partial differential equation governed by heat equations with mixed boundaries
	\begin{equation}\label{heat}
	\begin{cases}
	\displaystyle\frac{\partial u(t,x)}{\partial t}=\sigma  \displaystyle\frac{\partial^2 u(t,x)}{\partial x^2},&(t,x)\in \mathbb{R}^+\times [0,\pi];\\
	\displaystyle\frac{d u}{dx}(t,0)+k u(t,0)=0,\quad u(t,\pi)=0,&t\in \mathbb{R}^+;\\
	u(0,\cdot)=u_0\in L^2 [0,\pi],&
	\end{cases}
	\end{equation}
	where $k$ is a positive constant and $\sigma \in (0,+\infty)$ presents the thermal diffusivity. We select
	\begin{equation*}
	A_m=\sigma\frac{d^2}{dx^2},\quad Z=\{f\in H^2 (0,\pi):f(\pi)=0 \}.
	\end{equation*}
	The Dirichlet operator $D_\la$ is given by
	\begin{equation*}
	(D_\la u)(x)=\begin{cases}
	\displaystyle\frac{\sinh(\sqrt{\frac{\la}{\sigma}}(\pi-x))}{\sqrt{\frac{\la}{\sigma}}\cosh(\sqrt{\frac{\la}{\sigma}}\pi)}u, & \la >0, x\in [0,\pi]; \\
	(\pi-x) u , & \la=0, x\in [0,\pi].
	\end{cases}
	\end{equation*}
	From Example 2.6.8 \cite{Tucsnak Weiss}, the operator $A:=A_m$ with domain $D(A)=\{f\in Z:f'(0)=0 \}$ generates a positive and exponentially stable semigroup on $L^{2}[0,\pi]$. Since the triple $(A,-A_{-1}D_0,k\delta_0)$  is a regular linear system with $I:\mathbb{R}\longrightarrow \mathbb{R}$ is an admissible feedback operator (see Example 1 of \cite{Boulouz Bounit Driouich Hadd} and Lemma 3 of \cite{Boulouz Bounit Hadd}) then the solution of equation \eqref{heat} is positive. In addition, it is  uniformly exponentially stable if and only if $k<\frac{1}{\pi}$.
\end{example}
\section{Positive hyperbolic systems}
In this section, we apply the abstract results of stability obtained in the previous section to a special case of transport hyperbolic equations. To this end, let first introduce the following spaces and operators
\begin{itemize}
	\item The state space
	\begin{align}\label{state space}
	\mathcal{X}:=(L^2[0,1])^n\cong L^2([0,1],\mathbb{C}^n);
	\end{align}
	\item The boundary space of the value in endpoints
	\begin{align}\label{boundary space}
	\mathcal{U}:=\mathbb{C}^n;
	\end{align}
	\item The differential operator $\mathcal{Q}_{r,m}:\mathcal{Z}\longrightarrow \mathcal{X}$ given by
	\begin{align*}
	\mathcal{Q}_{r,m}:=\begin{pmatrix}
	-d_1 \frac{\partial }{\partial x}& &0\\
	& \ddots  &  \\
	0&  & -d_n \frac{\partial }{\partial x}
	\end{pmatrix}
	\end{align*}
	with (maximal) domain
	\begin{align*}
	\mathcal{Z}:=(H^{1}[0,1])^n\cong H^{1}([0,1],\mathbb{C}^n);
	\end{align*}
	\item The trace operator $\mathcal{G}:\mathcal{Z}\longrightarrow \mathcal{U}$ given by
	\begin{align*}
	\mathcal{G}:=\mathbb{I}_n \otimes \delta_0;
	\end{align*}
	\item The perturbed boundary $\mathcal{M}:\mathcal{Z}\longrightarrow \mathcal{U}$ given by
	\begin{align*}
	\mathcal{M}:=K\otimes \delta_1.
	\end{align*}
\end{itemize}
With these spaces and operators, the hyperbolic system \eqref{positive hyperbolic} can be viewed as the following perturbed boundary value system
\begin{align}\label{Q-r-system}
\begin{cases}
\dot{z}(t)=\mathcal{Q}_{r,m} z(t),&t\geq 0,\\
\mathcal{G}z(t)=\mathcal{M}z(t),&t\geq 0,\\
z(0)=\binom{y_0}{\varphi}.	&
\end{cases}
\end{align}
Now, we define the linear operator
\begin{align}\label{Q-cl-operator}
Q^{cl}_{r}:=Q_{r,m},\quad D(Q^{cl}_{r}):=\{f\in \mathcal{Z}: \mathcal{G}f=\mathcal{M}f\},
\end{align}
As shown in the previous section, to study the well-posedness and stability of the solution of the system \eqref{Q-r-system}, it suffices to prove that $Q^{cl}_{r}$ is a generator of a positive semigroup on $\mathcal{X}$. To that purpose, we first define
\begin{align*}
\mathcal{Q}_r:=\mathcal{Q}_{r,m},\quad D(\mathcal{Q}_r)=\{y=(y_1,y_2,...,y_n): y_i\in H^1[0,1],y_i(0)=0,i=1,...,n\}.
\end{align*}
Thus $\mathcal{Q}_r$ generates a nilpotent (hence positive and uniformly exponentially stable) semigroups $(\mathcal{R}(t))_{t\geq 0}$ which is given by
\begin{align*}
\mathcal{R}(t)=diag(\mathcal{R}_i(t))_{i=1,2,...,n}
\end{align*}
where
\begin{align*}
(\mathcal{R}_{i}(t)f_i)(s)=\begin{cases}
f_i(s-d_i t),&s\geq d_i t;\\
0,& \text{otherwise}
\end{cases}
\end{align*}
with $f_i\in L^2[0,1]$, $i=1,2,...,n$. Observe that $(\mathcal{R}_i(t))_{t\geq 0}$ is a right translated semigroup generated by $\mathcal{Q}_r^i:=d_i\frac{\partial}{\partial x}$ with its domain $D(\mathcal{Q}_r^i):=\{f\in H^1[0,1]:f(0)=0\}$. In addition
\begin{align*}
\mathcal{Q}_r:=\begin{pmatrix}
\mathcal{Q}_r^1& &0\\
& \ddots& \\
0 & &\mathcal{Q}_r^n
\end{pmatrix},\quad D(\mathcal{Q}_r):=D(\mathcal{Q}_r^1)\times \cdots \times D(\mathcal{Q}_r^n).
\end{align*}
It is clear that $\mathcal{G}$ is surjective. This means that  the assumptions ${\bf(H1)}$ and ${\bf(H2)}$ are satisfied. Therefore, the Dirichlet operator associated to $\mathcal{Q}_{r,m}$ and $\mathcal{G}$ is
\begin{align*}
(E_{\lambda}u)(x)=\begin{pmatrix}
e^{-\frac{\lambda}{d_1}x}& &0\\
& \ddots& \\
0& &e^{-\frac{\lambda}{d_n}x}
\end{pmatrix} u, \quad x\in [0,1],u\in \mathcal{U}.
\end{align*}
Define $\mathcal{B}=(\lambda-\mathcal{Q}_{r,-1})E_{\lambda},\quad \lambda \in  \mathbb{C},$ and $\mathcal{C}:=\mathcal{M}\imath$ where $\imath:D(\mathcal{Q}_r)\longrightarrow \mathcal{Z}$.

The following result presents the well-posedness of \eqref{positive hyperbolic}.
\begin{theorem}\label{positivity of hyperbolic}
	The operator $Q_r^{cl}$ defined in \eqref{Q-cl-operator} generates a $C_0$-semigroup $(\mathcal{R}^{cl}(t))_{t\geq 0}$ on $\mathcal{X}$. In particular, the problem \eqref{Q-r-system} is well-posed.
\end{theorem}
\begin{proof}
	According to our approach developed in the previous section, it suffices to show that $(\mathcal{Q}_{r},\mathcal{B},\mathcal{C})$ is a regular linear system with the identity operator $I:\mathcal{U}\longrightarrow \mathcal{U}$ as an admissible feedback. To this end, we define
	\begin{align*}
	(\tilde{\Phi}_t u)(s):=\begin{pmatrix}
	(\Phi^1_t u_1)(s)& &0\\
	& \ddots & \\
	0&   & (\Phi^n_t u_n)(s)
	\end{pmatrix}
	\end{align*}
	for $L^2(\mathbb{R}^+,\mathcal{U})$ and $s\in [0,1]$, where
	\begin{align*}
	(\Phi^i_t u_i)(s):=\begin{cases}
	u_i(t-\frac{s}{d_i}),& t\geq\frac{s}{d_i}\\
	0,& \text{otherwise,}
	\end{cases}\quad i=1,2,...,n
	\end{align*}
	for $u=(u_1,u_2,...,u_n)\in L^2(\mathbb{R}^+,\mathcal{U})$.
	By taking Laplace transform, we obtain
	\begin{align*}
	\widehat{(\tilde{\Phi}_\cdot u)}(\lambda)=E_{\lambda}\widehat{u}(\lambda).
	\end{align*}
	Then  $\tilde{\Phi}_t$ is a control map associated to $\mathcal{Q}_r$ and $\mathcal{B}$. It is clear that $\tilde{\Phi}_t u\in \mathcal{X}$ which means that $\mathcal{B}$ is an admissible control operator for $\mathcal{Q}_r$. On the other hand, let $\alpha>0$ we have
	\begin{align*}
	\int_{0}^{\alpha}|\delta_1 f_i (\cdot-d_i t)|^2 dt &= \int_{0}^{\alpha}|f_i(1-d_i t)|^2 dt\\
	& \leq \frac{1}{d_i}\int_{1-d_i \alpha}^{1}|f_i(\sigma)|^2 d\sigma\\
	& \leq \frac{1}{d_i} \|f_i\|_{L^2[0,1]}^2
	\end{align*}
	for $f_i\in D(\mathcal{Q}^i_{r})$ and $i=1,2,...,n$. Hence $\mathcal{C}_i$ is an admissible observation operator for $\mathcal{Q}^i_r$. This implies that $\mathcal{C}$ is an admissible observation operator for $\mathcal{Q}_r$. Let us now show that $(\mathcal{Q}_r,\mathcal{B},\mathcal{C})$ is well-posed on $\mathcal{X},\mathcal{U},$ and $\mathcal{U}$. In fact, for  $\alpha_i>\frac{1}{d_i}$,  $i=1,2,...,n$, we have
	\begin{align*}
	\int_{0}^{\alpha_i}|\delta_1 \Phi^i_t u_i|^2dt&=\int_{\frac{1}{d_i}}^{\alpha_i}|u_i (t-\frac{1}{d_i})|^2 dt\\
	&=\int_{0}^{\alpha_i -\frac{1}{d_i}}|u_i(\tau)|^2 d\tau\\
	& \leq \|u_i\|^2_{L^2[0,\alpha_i]}.
	\end{align*}
	Let $\alpha:=\max\{\alpha_i,i=1,2,...,n\}$. By a simple computation we obtain
	\begin{align*}
	\int_{0}^{\alpha}|M\tilde{\Phi}_t u|^2 dt &\leq k'\|u\|^2_{L^2([0,\alpha],U)}.
	\end{align*}
	where $k'$ is a positive constant.

	The next step is to show that $(\mathcal{Q}_r,\mathcal{B},\mathcal{C})$ is a regular linear system on $\mathcal{X},\mathcal{U},$ and $\mathcal{U}$. It suffices to verify that the limit of a function transfer associated to system $(\mathcal{Q}_r,\mathcal{B},\mathcal{C})$ exists. In fact, let $u=(u_1,u_2,...,u_n)\in \mathcal{U}$ and $\lambda\in \mathbb{C}$ then
	\begin{align*}
	\lim_{Re\lambda \to +\infty}H(\lambda)u&=\lim_{Re \lambda \to +\infty} K.\begin{pmatrix}
	e^{-\frac{\lambda}{d_1}}& &0\\
	& \ddots& \\
	0& &e^{-\frac{\lambda}{d_n}}
	\end{pmatrix}u\\
	& =\lim_{Re \lambda \to +\infty} \begin{pmatrix}
	\sum_{i=1}^{i=n}K_{1,i}e^{-\frac{\lambda}{d_i}}u_i\\
	\vdots \\
	\sum_{i=1}^{i=n}K_{n,i}e^{-\frac{\lambda}{d_i}}u_i
	\end{pmatrix}\\
	&=0 \quad (\text{ in } \mathcal{U}).
	\end{align*}
	Moreover, we have
	\begin{align*}
	\lim_{Re \lambda \to +\infty}\|H(\lambda)\|_{\mathcal{L}(\mathcal{U})}=0.
	\end{align*}
	This implies that there exists $\delta>0$ such that
	
	\begin{align*}
	\sup_{Re \lambda>\delta}\|H(\lambda)\|_{\mathcal{L}(\mathcal{U})}<\frac{1}{2}.
	\end{align*}
	Thus, $I-H(\lambda)$ is invertible and $\|(I-H(\lambda))^{-1}\|\leq 2$ for all $Re \lambda\geq \delta$. Hence, $I:\mathcal{U}\longrightarrow \mathcal{U}$ is an admissible feedback. This ends the proof.
\end{proof}
\begin{theorem}
	The following assertions hold.
	\begin{itemize}
		\item [{\rm (i)}] The semigroup $(\mathcal{R}^{cl}(t))_{t\ge 0}$ is positive.
		\item [{\rm (ii)}] The semigroup $(\mathcal{R}^{cl}(t))_{t\ge 0}$ is uniformly exponentially stable if and only if $r(K)<1$.
	\end{itemize}
\end{theorem}
\begin{proof}
	Since  $\displaystyle\lim_{\overset{\lambda \to \infty}{\lambda\in \mathbb{R}}}\|H(\lambda)\|=0$ then there exists $\lambda_0>0$ sufficiently large such that
	\begin{align*}
	r(H(\lambda_0)) <1.
	\end{align*}
	The assumption ${\bf(A2)}$ is then satisfied. Thus, by Theorem \ref{Positivity}, the semigroup $(\mathcal{R}^{cl}(t))_{t\geq 0}$ is positive. On the other hand, the transfer function at point $\lambda=0$ is
	\begin{align*}
	H(0):=K.
	\end{align*}
	Now as $s(\mathcal{Q}_r)<0$, by Theorem \ref{nece suff exp stablility}, the semigroup $(\mathcal{R}^{cl}(t))_{t\geq 0}$ is uniformly exponentially stable if and only if $r(K)<1$.
\end{proof}

\section{Hyperbolic systems with delay boundary conditions}

Consider the following hyperbolic systems with delay at endpoint $0$:
\begin{align}\label{delay heperbolic}
\begin{cases}
\vspace{0.1cm}
\displaystyle\frac{\partial y(x,t)}{\partial t}+D\displaystyle\frac{\partial y(x,t)}{\partial x}=0,& 0<x<1,t\geq 0,\\
y(0,t)=\displaystyle\int_{0}^{1}\int_{-1}^{0}d\mu(\theta)y_t(x,\theta)dx,&t\geq 0,\\
y(x,\theta)=\varphi(\theta),&0<x<1, -1\leq \theta\leq 0,\\
y(x,0)=y_0(x),&0<x<1.
\end{cases}
\end{align}
where $y(t,x)=(y_1(t,x),y_2(t,x),...,y_n(t,x))$, $D=diag\{d_1,d_2,...,d_n\}$ with the velocities $d_i>0$, for $i\in \{1,2,...,n \}$, $\mu:[-1,0]\to \mathcal{L}(\mathcal{U})_{+}$ is a positive operator-valued function of bounded variation with $\mu(0)=0$ and $y_t:[0,1]\times[-1,0]\to \mathcal{U}$ is a history function which is defined by $y_t(x,\theta):=y(x,t+\theta)$. Let $|\mu|$ be the positive Borel measure on $[-1,0]$ of the total variation of $\mu$.

It is well-known (e.g. in \cite{Hadd Idrissi Rhandi}), that the function $v(x,t,\theta):=y_t(x,\theta)=y(x,t+\theta)$, $(x,t,\theta)\in [0,1]\times [0,\infty)\times [-1,0]$, satisfies
\begin{align}\label{trans equ}
\begin{cases}
\displaystyle\frac{\partial v(x,t,\theta)}{\partial \theta}=\displaystyle\frac{\partial v(x,t,\theta)}{\partial \theta},&(x,t,\theta)\in [0,1]\times [0,+\infty)\times [-1,0],\\
v(x,t,0)=y(x,t),&(x,t)\in [0,1]\times [0,+\infty),\\
v(x,0,\cdot)=\varphi(x),&x\in[0,1].
\end{cases}
\end{align}
By using \eqref{trans equ}, we introduce the new state
\begin{align*}
z(x,t)=\binom{y(x,t)}{y_t(x,\cdot)},\quad 0\leq x\leq 1,t\geq 0,
\end{align*}
hence the problem \eqref{delay heperbolic} can be rewritten as
\begin{align*}
\begin{cases}
\displaystyle\frac{\partial z(x,t)}{\partial t}=\begin{pmatrix}
-D\displaystyle\frac{\partial }{\partial x}&0\\
0&\displaystyle\frac{\partial }{\partial \theta}
\end{pmatrix}z(x,t),&0<x<1,t\geq 0,\\
\begin{pmatrix}
y(x,0)\\
y_t(x,0)
\end{pmatrix}=\begin{pmatrix}
\displaystyle\int_{0}^{1}\int_{-1}^{0}d\mu(\theta)y_t(x,\theta)dx\\
y(x,t)
\end{pmatrix},&0<x<1,t\geq 0,\\
z(x,0)=\begin{pmatrix}
y_0(x)\\
\varphi(x,\cdot)
\end{pmatrix},&0<x<1.
\end{cases}
\end{align*}
In the sequel, we introduce some spaces and operators. We define the space
\begin{align*}
\mathbb{x}:=L^{2}([-1,0],\mathcal{X})\cong L^{2}([-1,0]\times [0,1],\mathbb{C}^n)
\end{align*}
and the maximal operator
\begin{align*}
\mathcal{Q}_{l,m}:=\begin{pmatrix}
\frac{d}{d\theta}& &0\\ &\ddots& \\
0& &\frac{d}{d\theta}
\end{pmatrix}, \quad
D(\mathcal{Q}_{l,m})=W^{1,2}([-1,0],\mathcal{X})\cong W^{1,2}([-1,0]\times  [0,1],\mathbb{C}^n):=\mathbb{Z}.
\end{align*}
On the other hand, we define
\begin{align*}
N\varphi=\int_{0}^{1}\int_{-1}^{0}d\mu(\theta)\varphi(x,\theta)dx,\quad \varphi\in \mathbb{Z}.
\end{align*}
In addition, we introduce the product spaces
\begin{align*}
\mathfrak{X}=\mathcal{X}\times \mathbb{X},\quad \mathfrak{Z}=\mathcal{Z}\times \mathbb{Z},\quad \mathfrak{U}=\mathcal{U}\times \mathcal{X},
\end{align*}
the differential operator
\begin{align*}
\mathfrak{A}_m=\begin{pmatrix}
\mathcal{Q}_{r,m}&0\\0&\mathcal{Q}_{l,m}
\end{pmatrix}:\mathfrak{Z}\longrightarrow \mathfrak{X},
\end{align*}
and the boundary operators
\begin{align*}
\mathfrak{G}=\begin{pmatrix}
\mathcal{G}&0\\0&\mathcal{G}_0
\end{pmatrix}:\mathfrak{Z}\longrightarrow \mathfrak{U},\quad \mathfrak{M}=\begin{pmatrix}
0&N\\I&0
\end{pmatrix}:\mathfrak{Z}\longrightarrow \mathfrak{U},
\end{align*}
where $\mathcal{G}_0:=\mathbb{I}_n \otimes \delta_0:\mathbb{Z}\longrightarrow \mathcal{X}$. Hence the problem \eqref{delay heperbolic} becomes
\begin{align*}
\begin{cases}
\dot{z}(t)=\mathfrak{A}_m z(t),&t\geq 0,\\
\mathfrak{G}z(t)=\mathfrak{M}z(t),&t\geq 0,\\
z(0)=\binom{y_0}{\varphi}.	&
\end{cases}
\end{align*}
We recall that the opeartor
\begin{align*}
\mathcal{Q}_l:=\mathcal{Q}_{l,m},\quad D(\mathcal{Q}_l):=\ker \mathcal{G}_0
\end{align*}
generates a $C_{0}$-semigroup $(\mathcal{S}(t))_{t\geq 0}$ given by
\begin{align*}
\mathcal{S}(t)=\begin{pmatrix}
S^{1}(t)& &0\\ &\ddots& \\
0& &S^{n}(t)
\end{pmatrix}
\end{align*}
where
\begin{align*}
(S^{i}(t)g_i)(\theta)=\begin{cases}
g_i(t+\theta), &\text{ if } t+\theta\leq 0,\\
0,&\text{ otherwise,}
\end{cases}   \quad g_i\in L^{2}([-1,0]),\quad i=1,2,...,n
\end{align*}
is a left semigroup generated by
\begin{align*}
\mathcal{Q}^i_l=\frac{d}{d\theta},\quad D(\mathcal{Q}^i_l)=\{h\in W^{1,2}([-1,0]):h(0)=0\}.
\end{align*}
We now define the operator
\begin{align*}
\mathfrak{A}=\mathfrak{A}_m,\quad D(\mathfrak{A}):=\ker\mathfrak{G}.
\end{align*}
Then
\begin{align*}
\mathfrak{A}=\begin{pmatrix}
Q_r&0\\0&Q_l
\end{pmatrix}, \quad D(\mathfrak{A})=D(Q_r)\times D(Q_l)
\end{align*}
generates a $C_0$-semigroup $(\mathfrak{T}(t))_{t\geq 0}$ on $\mathfrak{X}$ given by
\begin{align*}
\mathfrak{T}(t)=\begin{pmatrix}
\mathcal{R}(t)&0\\0&\mathcal{S}(t)
\end{pmatrix},\quad t\geq 0.
\end{align*}
Hence the assumptions ${\bf(H1)}$ and ${\bf(H2)}$ are satisfied. According to \cite{Greiner}, the Dirichlet operator associated to $\mathfrak{A}_m$ and $\mathfrak{G}$ is
\begin{align}\label{Dirichlet bb}
\mathfrak{D}_{\lambda}&=(\mathfrak{G}_{|\ker(\lambda-\mathfrak{A}_m)})^{-1}\nonumber\\
&=\begin{pmatrix}
E_\lambda&0\\
0&e_{\lambda}
\end{pmatrix},
\end{align}
where $e_{\lambda}:\mathcal{X}\longrightarrow \mathbb{X}$ is defined by $(e_{\lambda}\phi)(\theta,x)=e^{\lambda \theta}\phi(x)$ for all $\phi \in \mathcal{X},\theta\in[-1,0]$ and $[0,1]$.
We set
\begin{align*}
\mathfrak{B}=(\lambda-\mathfrak{A}_{-1})\mathfrak{D}_{\lambda},\quad \text{and} \quad \mathfrak{C}=\mathfrak{M}_{|D(\mathfrak{A})}.
\end{align*}
Then
\begin{align*}
\mathfrak{B}=\begin{pmatrix}
\mathcal{B}&0\\
0&\beta
\end{pmatrix}, \quad \text{and} \quad \mathfrak{C}=\begin{pmatrix}
0&N_0\\
I&0
\end{pmatrix},
\end{align*}
where $N_0$ is the restriction of $N$ to $D(Q_l)$ and $\beta=(-\mathcal{Q}_{l,-1})e_0$.

From \cite{Hadd Idrissi Rhandi}, we deduce that the triple $(\mathcal{Q}_l,\beta,N_0)$ is a regular linear system on $\mathbb{X},\mathcal{X}$ and $\mathcal{X}$ with $I:\mathcal{X}\longrightarrow \mathcal{X}$ is an admissible feedback operator.

The following result shows that the control system $(\mathfrak{A},\mathfrak{B})$ is well-posed.
\begin{lemma}
	$\mathfrak{B}$ is an admissible control operator for $\mathfrak{A}$.
\end{lemma}
\begin{proof}
	Let $t_0>0$ and $\binom{u}{v}\in L^{2}(\mathbb{R}_+,\mathfrak{U})$, we have
	\begin{align}\label{control map bb}
	\check{\Phi}_{t_0}\binom{u}{v}&=\int_{0}^{t_0}\mathfrak{T}_{-1}(t_0-s)\mathfrak{B}\binom{u(s)}{v(s)}ds\nonumber\\
	&=\int_{0}^{t_0}\begin{pmatrix}
	R_{-1}(t_0-s)&0\\0&S_{-1}(t_0-s)
	\end{pmatrix}\binom{\mathcal{B}u(s)}{\beta v(s)}ds\nonumber\\
	&=\begin{pmatrix}
	\tilde{\Phi}_{t_0}&0\\0&\bar{\Phi}_{t_0}
	\end{pmatrix}\binom{u}{v}
	\end{align}
	where  $\tilde{\Phi}_{t_0}$ is a control map associated to $\mathcal{Q}_r$ and $\mathcal{B}$ and $\bar{\Phi}_{t_0}$ is a control map associated to $\mathcal{Q}_l$ and $\beta$.
	
	Since $\tilde{\Phi}_{t_0}\in \mathcal{X}$ and $\bar{\Phi}_{t_0} \in \mathbb{X}$ then $\check{\Phi}_{t_0}\in \mathfrak{X}$ which means that 	$\mathfrak{B}$ is an admissible control operator for $\mathfrak{A}$.
\end{proof}
The following lemma present the well-posedness of the observation system $(\mathfrak{C},\mathfrak{A})$. We give the similar proof as shown in \cite[Example 5.2]{Hadd Manzo Ghandi}.
\begin{lemma}
	$\mathfrak{C}$ is an admissible observation operator for $\mathfrak{A}$.
\end{lemma}
\begin{proof}
	Let $f\in D(Q_r),g\in D(Q_l)$ and $0<\alpha<1$, we have
	\begin{align*}
	\displaystyle\int_{0}^{\alpha}\left\|\mathfrak{C}\mathfrak{T}(t)\binom{f}{g}\right\|_{\mathfrak{U}}^{2}dt&=\displaystyle\int_{0}^{\alpha}(\|\mathcal{R}(t)f\|_{\mathcal{X}}+\|N_0 \mathcal{S}(t)g\|_{\mathcal{U}})^2 dt\\
	&\leq 2\int_{0}^{\alpha}(\|\mathcal{R}(t)f\|^2_{\mathcal{X}}+\|N_0 \mathcal{S}(t)g\|^2_{\mathcal{U}})dt\\
	&\leq k(\alpha)\|f\|^2_{\mathcal{X}}+2\int_{0}^{\alpha}\|N_0 \mathcal{S}(t)g\|^{2}_{\mathcal{U}}dt
	\end{align*}
	for a constant $k(\alpha)>0$. According to H\"{o}lder's inequality  and Fubini's Theorem, we obtain
	\begin{align*}
	\int_{0}^{\alpha}\|N_0 S(t)g\|^{2}_{\mathcal{U}}dt&=\int_{0}^{\alpha}\left\|\int_{0}^{1}\int_{-1}^{0}d\mu(\theta)(S(t)g)(x,\theta)dx\right\|^2_{\mathcal{U}}dt\\
	&=\int_{0}^{\alpha}\left\|\int_{0}^{1}\int_{-1}^{-t}d\mu(\theta)g(x,t+\theta)dx\right\|^2_{\mathcal{U}}dt\\
	&\leq \int_{0}^{\alpha}\left(\int_{0}^{1}\int_{-1}^{-t}\|g(x,t+\theta)\|_{\mathcal{U}}d|\mu|(\theta)dx\right)^2dt\\
	&\leq |\mu|([-1,0])\int_{0}^{\alpha}\int_{0}^{1}\int_{-1}^{-t}\|g(x,t+\theta)\|^2_{\mathcal{U}}d|\mu|(\theta)dxdt\\
	&=|\mu|([-1,0])\int_{0}^{\alpha}\int_{-1}^{-t}\int_{0}^{1}\|g(x,t+\theta)\|^2_{\mathcal{U}}d|dx\mu|(\theta)dt\\
	&=|\mu|([-1,0])\int_{0}^{\alpha}\int_{-1}^{-t}\|g(\cdot,t+\theta)\|^2_{\mathcal{X}}d|\mu|(\theta)dt\\
	&=|\mu|([-1,0])\int_{-1}^{0}\int_{0}^{-\theta}\|g(\cdot,t+\theta)\|^2_{\mathcal{X}}dtd|\mu|(\theta)\\
	&\leq (|\mu|([-1,0]))^2\|g\|^2_{\mathbb{X}}.
	\end{align*}
	Therefore
	\begin{align*}
	\displaystyle\int_{0}^{\alpha}\left\|\mathfrak{C}\mathfrak{T}(t)\binom{f}{g}\right\|_{\mathfrak{U}}^{2}dt\leq c^2\left\|\binom{f}{g}\right\|_{\mathfrak{X}}^2,
	\end{align*}
	where $c:=\max\{\sqrt{k(\alpha)},|\mu|([-1,0])\}$. The proof is claimed.
\end{proof}

The following result gives the well-posedness of the problem \eqref{delay heperbolic}.
\begin{proposition}\label{Prop 5.3}
	The operator $\mathfrak{A}^{cl}$ generates a $C_0$-semigroup $(\mathfrak{T}^{cl}(t))_{t\geq 0}$ on $\mathfrak{X}$.
\end{proposition}
\begin{proof}
	In order to prove the generation property of the operator $\mathfrak{A}^{cl}$, we shall show that the triplet $(\mathfrak{A},\mathfrak{B},\mathfrak{C})$ is a regular linear system on $\mathfrak{X},\mathfrak{U}$ and $\mathfrak{U}$ with $I:\mathfrak{U}\longrightarrow \mathfrak{U}$ is an admissible feedback operator. We first show that $(\mathfrak{A},\mathfrak{B},\mathfrak{C})$ is well-posed. To do this, we need to define input-ouput operators. In fact, le $\lambda>0$ be a sufficiently large, we have
	\begin{align*}
	\mathfrak{C}\lambda R(\lambda,\mathfrak{A})&=\lambda \begin{pmatrix}
	0&N_0\\I&0
	\end{pmatrix}\begin{pmatrix}
	R(\lambda,Q_r)&0\\0&R(\lambda,Q_l)
	\end{pmatrix}\\
	&=\begin{pmatrix}
	0&N_0\lambda R(\lambda,Q_l)\\\lambda R(\lambda,Q_r)&0
	\end{pmatrix}
	\end{align*}
	Thus the Yosida extension $\mathfrak{C}_{\Lambda}$ of $\mathfrak{C}$ for $\mathfrak{A}$ is
	\begin{align*}
	\mathfrak{C}_{\Lambda}=\begin{pmatrix}
	0&N_{0,\Lambda}\\I&0
	\end{pmatrix},\quad D(\mathfrak{C}_{\Lambda})=\mathcal{X}\times D(N_{0,\Lambda})
	\end{align*}
	where $N_{0,\Lambda}$ is the Yosida extension of $N_0$ for $Q_l$.

	By using \eqref{control map bb} and the facts that $\mathcal{B}$ is an admissible control operator for $Q_r$ and $(Q_l,\beta,N_0)$ is a regular linear system, we have
	\begin{align*}
	\text{Range } \check{\Phi}_{t}\subset D(\mathfrak{C}_{\Lambda}),\quad \text{a.e. }t\geq 0.
	\end{align*}
	This permit us to define the input-output operator $\check{\mathbb{F}}$  associated to system $(\mathfrak{A},\mathfrak{B},\mathfrak{C})$ as follows
	\begin{align*}
	\check{\mathbb{F}}&:=\mathfrak{C}_{\Lambda}\check{\Phi}_{\cdot}=\begin{pmatrix}
	0&\bar{\mathbb{F}}\\
	\tilde{\mathbb{F}}&0
	\end{pmatrix},
	\end{align*}
	where $\bar{\mathbb{F}}$ is the input-output operator associated to system $(\mathcal{Q}_l,\beta,N_0)$ and  $\tilde{\mathbb{F}}$ is the input-output operator   associated to system $(\mathcal{Q}_r,\mathcal{B},I)$. It follows that the triplet $(\mathfrak{A},\mathfrak{B},\mathfrak{C})$ is a well-posed system.  On the other hand, observe that $\text{Range } \mathfrak{D}_{\lambda}\subset D(\mathfrak{C}_{\Lambda})$ which means by Theorem \ref{characterization regular} that the triple $(\mathfrak{A},\mathfrak{B},\mathfrak{C})$ is a regular linear system. On the other hand, the transfer function associated to $(\mathfrak{A},\mathfrak{B},\mathfrak{C})$  is given by
	\begin{align*}
	\check{H}(\lambda):=\mathfrak{M}\mathfrak{D}_{\lambda}=\mathfrak{C}_{\Lambda}\mathfrak{D}_{\lambda}=\begin{pmatrix}
	0&\bar{H}(\lambda)\\ \tilde{H}(\lambda)&0
	\end{pmatrix},\quad \lambda\in \mathbb{C},
	\end{align*}
	where $\bar{H}(\lambda):=Ne_{\lambda}$ is a transfer function associated to $(\mathcal{Q}_l,\beta,N_0)$ and $\tilde{H}(\lambda):=E_{\lambda}$ is a transfer function associated to $(\mathcal{Q}_r,\mathcal{B},I)$. Observe that
	\begin{align*}
	\lim_{Re\lambda\to+\infty}\|\tilde{H}(\lambda)\|=0.
	\end{align*}
	On the other hand, for an arbitrary $0<\epsilon<1$ and by a simple computation we have
	\begin{align*}
	\|\bar{H}(\lambda)\|\leq e^{-\epsilon Re \lambda}|\mu|([-1,0])+|\mu|([-\epsilon,0]).
	\end{align*}
	Since $|\mu|([-\epsilon,0])\longrightarrow 0$ as $\epsilon\to 0$ then
	\begin{align*}
	\lim_{Re \lambda \to +\infty}\|\bar{H}(\lambda)\|=0.
	\end{align*}
	We conclude that
	\begin{align*}
	\lim_{Re \lambda \to +\infty}\|\check{H}(\lambda)\|=0.
	\end{align*}
	Thus, there exist $\gamma>0$ such that
	\begin{align*}
	\sup_{Re\lambda\geq \gamma}\|\check{H}(\lambda)\|<\frac{1}{2},
	\end{align*}
	which implies that $\|(I-\check{H}(\lambda))^{-1}\|\leq 2$ for all $Re\lambda\geq \gamma$. Hence $I$ is an admissible feedback operator. Thus, the desired result is obtained.
\end{proof}
\begin{theorem}\label{Thm stab delay hyp}
	The semigroup $(\mathfrak{T}^{cl}(t))_{t\geq 0}$ generated by $\mathfrak{A}^{cl}$ is positive. Moreover, the semigroup $(\mathfrak{T}^{cl}(t))_{t\geq 0}$ is uniformly exponentially stable if and only if $r(Ne_0E_0)<1$.
\end{theorem}
\begin{proof}
	We first mention that the assumptions of Theorem \ref{Positivity} are satisfied and hence the semigroup $(\mathfrak{T}^{cl}(t))_{t\geq 0}$ is positive if there exists $\lambda_0>0$ such that
	\begin{align*}
	r(\check{H}(\lambda_0))<1,
	\end{align*}
	where $\check{H}(\lambda_0):=\mathfrak{M}\mathfrak{D}_{\lambda_0}$. On the other hand, let $\lambda\in \mathbb{C}$, we have
	\begin{align*}
	I-\check{H}(\lambda)=\begin{pmatrix}
	I&-\bar{H}(\lambda)\\ -\tilde{H}(\lambda)&I
	\end{pmatrix},
	\end{align*}
	where $\bar{H}(\lambda):=Ne_\lambda$ and $\tilde{H}(\lambda):=E_{\lambda}$. Using the facts that
	\begin{align*}
	\lim_{Re\lambda\to +\infty}\|\tilde{H}(\lambda)\|=\lim_{Re\lambda\to +\infty}\|\bar{H}(\lambda)\|=0,
	\end{align*}
	we then obtain $1\in\rho(\check{H}(\lambda))\Longleftrightarrow 1\in \rho(\tilde{H}(\lambda)\bar{H}(\lambda))\Longleftrightarrow 1\in \rho(\bar{H}(\lambda)\tilde{H}(\lambda))$ and
	\begin{align*}
	(I-\check{H}(\lambda))^{-1}=\begin{pmatrix}
	(I-\bar{H}(\lambda)\tilde{H}(\lambda))^{-1}&(I-\bar{H}(\lambda)\tilde{H}(\lambda))^{-1}\bar{H}(\lambda)\\
	(I-\tilde{H}(\lambda)\bar{H}(\lambda))^{-1}\tilde{H}(\lambda)&(I-\tilde{H}(\lambda)\bar{H}(\lambda))^{-1}
	\end{pmatrix}.
	\end{align*}
	We deduce that $$r(\check{H}(\lambda))<1\Longleftrightarrow r(\tilde{H}(\lambda)\bar{H}(\lambda))<1\Longleftrightarrow r(\tilde{H}(\lambda)\bar{H}(\lambda))<1.$$
	Since $\displaystyle\lim_{Re\lambda\to +\infty}\|\tilde{H}(\lambda)\bar{H}(\lambda)\|=\displaystyle\lim_{Re\lambda\to +\infty}\|\bar{H}(\lambda)\tilde{H}(\lambda)\|=0$, then there exists $\lambda_0>0$ sufficiently large such that
	\begin{align*}
	r(\bar{H}(\lambda_0)\tilde{H}(\lambda_0))<1.
	\end{align*}
	It follows that the semigroup $(\mathfrak{T}^{cl}(t))_{t\geq 0}$ is positive. Moreover, one can see that $s(\mathfrak{A})<0$ and according to Theorem \ref{nece suff exp stablility}, we deduce that the semigroup $(\mathfrak{T}^{cl}(t))_{t\geq 0}$ is uniformly exponentially stable if and only if $r(Ne_0E_0)<1$.
\end{proof}
The following result shows that a small delay effect on the stability of system \eqref{delay heperbolic}. The proof is directly obtained from Theorem \ref{Thm stab delay hyp}.
\begin{corollary}
	If $|\mu|([-1,0])<1$ then the semigroup $(\mathfrak{T}^{cl}(t))_{t\geq 0}$ is uniformly exponentially stable on $\mathfrak{X}$.
\end{corollary}
\begin{example}
	Let $L\in\mathbb{R}^{n\times n}_+$ and consider
	\begin{align*}
	\mu(\theta)=\begin{cases}
	L&\text{ if }\theta=-1;\\
	0&\text{ if }\theta\in (-1,0].
	\end{cases}
	\end{align*}
	Thus the problem \ref{delay heperbolic} becomes
	\begin{align}\label{delay exam}
	\begin{cases}
	\vspace{0.1cm}
	\displaystyle\frac{\partial y(x,t)}{\partial t}+D\displaystyle\frac{\partial y(x,t)}{\partial x}=0,& 0<x<1,t\geq 0,\\
	y(0,t)=L\displaystyle\int_{0}^{1}y(x,t-1)dx,&t\geq 0,\\
	y(x,\theta)=\varphi(\theta),&0<x<1, -1\leq \theta\leq 0,\\
	y(x,0)=y_0(x),&0<x<1.
	\end{cases}
	\end{align}
	
	Due to Proposition \ref{Prop 5.3} and Theorem \ref{Thm stab delay hyp}, the problem \eqref{delay exam} admits a unique positive solution and it is uniformly exponential stability if and only if $r(L)<1$.	
	
\end{example}

\end{document}